\documentclass[12pt,reqno]{amsart}
\usepackage{amsmath,amsxtra,latexsym,amsthm,amssymb,amscd,pb-diagram}
\usepackage{mathabx,mathrsfs,amsfonts}
\usepackage[colorlinks=true,linkcolor=magenta,citecolor=blue]{hyperref}
\usepackage[margin=1in]{geometry}

\usepackage{graphicx}
\usepackage{epsfig}
\usepackage{color}
\usepackage{enumitem}

\newtheorem{dn}{Definition}[section]
\newtheorem{dl}{Theorem}[section]
\newtheorem{md}{Proposition}[section]

\newtheorem{bd}{Lemma}[section]

\newtheorem{hq}{Corollary}[section]
\newtheorem{nx}{Remark}[section]

\newtheorem{vd}{Example}[section]
\newcommand{\R}{\mathbb{R}}
\newcommand{\Z}{\mathbb{Z}}

\newcommand{\ity}{\infty}

\newcommand{\bbd}{\begin{bd}}
\newcommand{\ebd}{\end{bd}}
\newcommand{\bdn}{\begin{dn}}
\newcommand{\edn}{\end{dn}}
\newcommand{\bhq}{\begin{hq}}
\newcommand{\ehq}{\end{hq}}
\newcommand{\bdl}{\begin{dl}}
\newcommand{\edl}{\end{dl}}
\newcommand{\bnx}{\begin{nx}}
\newcommand{\enx}{\end{nx}}
\newcommand{\bmd}{\begin{md}}
\newcommand{\emd}{\end{md}}
\newcommand{\bvd}{\begin{vd}}
\newcommand{\evd}{\end{vd}}

\title[Critical exponent for structurally damped waves of derivative type]{Critical exponent for semi-linear structurally damped wave equation of derivative type}
\author{Tuan Anh Dao}
\address{Tuan Anh Dao \hfill\break
$\quad$ School of Applied Mathematics and Informatics, Hanoi University of Science and Technology, No.1 Dai Co Viet road, Hanoi, Vietnam \hfill\break
Faculty for Mathematics and Computer Science, TU Bergakademie Freiberg, Pr\"{u}ferstr. 9, 09596, Freiberg, Germany}
\email{anh.daotuan@hust.edu.vn}

\author{Ahmad Z. Fino}
\address{Ahmad Z. Fino \hfill\break
Department of Mathematics, Faculty of Sciences, Lebanese University, P.O. Box 826, Tripoli, Lebanon}
\email{ahmad.fino01@gmail.com; afino@ul.edu.lb}

\begin{document}
\date{March 02, 2020}

\begin{abstract}
Main purpose of this paper is to study the following semi-linear structurally damped wave equation with nonlinearity of derivative type:
$$u_{tt}- \Delta  u+ \mu(-\Delta)^{\sigma/2} u_t= |u_t|^p,\quad u(0,x)= u_0(x),\quad u_t(0,x)=u_1(x),$$
with $\mu>0$, $n\geq1$, $\sigma \in (0,2]$ and $p>1$. In particular, we are going to prove the non-existence of global weak solutions by using a new test function and suitable sign assumptions on the initial data in both the subcritical case and the critical case.
\end{abstract}

\maketitle

\section{Introduction} \label{Sec.Intro}
This paper is concerned with the Cauchy problem for semi-linear structurally damped wave equation with the power nonlinearity of derivative type (powers of the first order time-derivatives of solutions as nonlinear terms) as follows:
\begin{equation} \label{equation1}
\begin{cases}
u_{tt}- \Delta  u+ \mu(-\Delta)^{\sigma/2} u_t= |u_t|^p, &\quad x\in \R^n,\, t> 0, \\
u(0,x)= u_0(x),\quad u_t(0,x)=u_1(x), &\quad x\in \R^n,
\end{cases}
\end{equation}
where $\mu>0$, $\sigma \in (0,2]$, $n\geq1$ and $p>1$. Here $(-\Delta)^{\sigma/2}$ is the fractional Laplacian defined as in Definition \ref{Def.FracLaplace} below when $\sigma\in(0,2)$, and when $\sigma=2$ it is the classical Laplacian.

Our main goal is to investigate the so-called critical exponent for \eqref{equation1}. By critical exponent $p_c=p_c(n,\sigma)$ we mean that global (in time) solutions cannot exist (it sometimes called blow-up is some cases), under suitable sign assumption on the initial data, in the critical and subcritical cases $p\leq p_c$, whereas small data global (in time) solutions exist in the supercritical case $p>p_c$. 

Regarding the structurally damped wave equation \eqref{equation1} with the power nonlinearity $|u|^p$, the critical exponent has been investigated by D'Abbicco and Reissig \cite{DabbiccoReissig}, where they proposed to distinguish between ``parabolic like models" in the case $\sigma \in (0,1]$, the so-called effective damping, and ``hyperbolic like models"in the remaining case $\sigma \in (1,2]$, the so-called non-effective damping according to expected decay estimates (see more \cite{DabbiccoEbert2016}). In the former case, they proved the existence of global (in time) solutions when
$$p>p_0(n,\sigma):=1+ \frac{2}{(n-\sigma)_+}$$
for the small initial data and low space dimensions $2\leq n\leq 4$ by using the energy estimates. Here we denote $(r)_+:= \max\{r,0\}$ as its positive part for any $r\in \R$. Afterwards, D'Abbicco and Ebert \cite{DabbiccoEbert2014} extended their global existence results to higher space dimensions by using $L^r-L^q$ estimates for solutions to the corresponding linear equation. On the other hand, the authors indicated in \cite{DabbiccoReissig} the non-existence of global (in time) solutions, just when $\sigma=1$, if the condition
$$p\leq p_0(n,1)= 1+ \frac{2}{n-1}$$
holds by using the standard test function method via the non-negativity of the fundamental solution (see also \cite{DuongKainaneReissig}). In these cited papers, one should recognizes that the assumptions
$$u_0=0 \quad \text{ and }\quad u_1\ge 0 $$
come to guarantee the non-negativity of the fundamental solution, which cannot be expected for any $\sigma\in(0,2]$. Quite recently, the global non-existence result for any $\sigma\in(0,2]$ has been completed by Dao and Reissig \cite{DaoReissig} when $p\leq p_0(n,\sigma)$ and for all $n\geq1$ by using a modified test function which deals with sign-changing data condition, namely
$$u_0=0 \quad \text{ and }\quad u_1 \in L^1 \text{ satisfying }\int_{\R^n}u_1(x)dx>0. $$
Again, we can see that assuming the first data $u_0=0$ is necessary to require. It seems that the previous used approaches do not work so well if $u_0$ is not identically zero. For the non-effective case $\sigma\in(1,2]$, the global existence results were also shown by \cite{DabbiccoReissig} only for $p>1+(1+\sigma)/(n-1)$ with $n\geq2$, while the blow-up of solutions has been obtained by \cite{DaoReissig} when $p\leq1+2/(n-1)$. Unfortunately, there appears a gap between the two exponents $1+(1+\sigma)/(n-1)$ and $1+2/(n-1)$. This is naturally due to the hyperbolic-like structure of the problem which seems not suitable with the standard test function method used to prove blow-up results.

Let us come back our interest to consider the structurally damped wave equation with the power nonlinearity of derivative type \eqref{equation1}. At present, there do not seem to be so many related manuscripts. D'Abbicco and Ebert \cite{DabbiccoEbert2017} proved the global (in time) existence of small data solutions for any
$$p>p_1(n,\sigma):= 1+\sigma/n $$
in the case of $\sigma\in(0,1)$ and lower space dimensions, as well as for any
$$p>p_1(n,1)= 1+1/n $$
in the case of $\sigma=1$ and all $n\geq1$. For the purpose of looking for the global (in time) existence of small data Sobolev solutions to \eqref{equation1}, with $\sigma\in(0,1)$, from suitable function spaces basing on $L^q$ spaces, with $q\in (1,\infty)$, we address the interested readers to the new papers of Dao and Reissig \cite{DaoReissig1}). When $\sigma\in(1,2]$, the only global existence results known up to our knowledge can be found in \cite{DaoReissig2} for any $p>\bar{p}$, where $\bar{p}$ are a suitable exponent, under small initial data in Sobolev space. From these observations, it still keeps an open problem so far to indicate a non-existence result for \eqref{equation1} in all cases $\sigma\in(0,2]$.

For this reason, our main motivation of this paper is to fill this lack. Especially, we would like to face up to dealing with the fractional Laplacian $(-\Delta)^{\sigma/2}$, the well-known nonlocal operators, where $\sigma$ is supposed to be a fractional number in $(0,2)$. As we can see, this case was not included in \cite{DabbiccoEbert2017} since the standard test function method seems difficult to be directly applied to these fractional Laplacian. To overcome this difficulty, the application of a new modified test function developed by Dao and Fino in the recent work \cite{DaoFino}, and mentioned in \cite{BonforteVazquez}, comes into play. Moreover, as analyzed above, we want to point out that it is challenging to follow the recent papers (\cite{DabbiccoReissig,DaoReissig}) in terms of the treatment of $u_0\neq 0$. Hence, the other point worthy of noticing in the present paper is that our method can be applicable effectively to relax the limitation of the assumption for $u_0=0$, which plays an important role in the proofs of blow-up results in severval previous literatures (see, for example, \cite{DabbiccoEbert2017,DuongKainaneReissig,DabbiccoReissig,DaoReissig}). \medskip

\textbf{Notations}
\begin{itemize}[leftmargin=*]
\item We denote the constant $\bar{\sigma}:=\min\{\sigma,1\}$, where $\sigma \in (0,2]$.
\item For later convenience, $C$ and $C_i$ with $i \in \Z$ stand for suitable positive constants.
\item For given nonnegative $f$ and $g$, we write $f\lesssim g$ if $f\le Cg$. We write $f \approx g$ if $g\lesssim f\lesssim g$.
\end{itemize}

Our main result reads as follows.
\begin{dl}[\textbf{Blow-up}] \label{Blow-up.Main}
Let $\sigma \in (0,2]$. We assume that $(u_0,u_1)\in \big(L^1(\mathbb{R}^n)\cap L^2(\mathbb{R}^n)\big)^2$ satisfying the following condition:
\begin{equation} \label{blow-up.1}
\int_{\R^n} u_1(x)dx > 0.
\end{equation}
If
\begin{equation} \label{blow-up.2}
p \in \Big(1, 1+ \frac{\bar{\sigma}}{n}\Big],
\end{equation}
then, there is no global (in time) weak solution to \eqref{equation1}.
\end{dl}

\bnx
\fontshape{n}
\selectfont
From the statement of Theorem \ref{Blow-up.Main}, we want to underline that the sign-changing data assumption of $u_0$ brings a new contribution in comparison with the previous studies \cite{DabbiccoEbert2017,DuongKainaneReissig,DabbiccoReissig,DaoReissig}. 
\enx

\bnx
\fontshape{n}
\selectfont
By Theorem \ref{Blow-up.Main} and Theorem $7$ in \cite{DabbiccoEbert2017}, it is clear that $p_c:=1+\sigma/n$ is the critical exponent of \eqref{equation1} when $\sigma \in (0,1]$. It is still an open problem whether $1+1/n$ is the critical exponent of \eqref{equation1} when $\sigma \in (1,2]$.
\enx

\section{Preliminaries} \label{Sec.Pre}
In this section, we collect some preliminary knowledge needed in our proofs.
\bdn[\cite{Kwanicki}]\label{Def.FracLaplace}
\fontshape{n}
\selectfont
Let $s \in (0,1)$. Let $X$ be a suitable set of functions defined on $\R^n$. Then, the fractional Laplacian $(-\Delta)^s$ in $\R^n$ is a non-local operator given by
$$ (-\Delta)^s: \,\,v \in X  \to (-\Delta)^s v(x):= C_{n,s}\,\, p.v.\int_{\R^n}\frac{v(x)- v(y)}{|x-y|^{n+2s}}dy $$
as long as the right-hand side exists, where $p.v.$ stands for Cauchy's principal value, $C_{n,s}:= \frac{4^s \Gamma(\frac{n}{2}+s)}{\pi^{\frac{n}{2}}\Gamma(-s)}$ is a normalization constant and $\Gamma$ denotes the Gamma function.
\edn

\bbd\label{lemma1}\cite[Lemma~2.3]{DaoFino}
Let $\langle x\rangle:=(1+(|x|-1)^4)^{1/4}$ for all $x\in\mathbb{R}^n$. Let $s \in (0,1]$ and $\phi:\mathbb{R}^n\rightarrow \mathbb{R}$ be a function defined by
\begin{equation}\label{testfunction}
\phi(x)=\left\{
\begin{array}{ll}
1&\hbox{if}\;\;|x|\leq1,\\
{}\\
\langle x\rangle^{-n-2s}&\hbox{if}\;\;|x|\geq1.
\end{array}
\right.
\end{equation}
Then, $\phi\in C^2(\mathbb{R}^n)$ and the following estimate holds:
\begin{equation}\label{10}
\left|(-\Delta)^s\phi(x)\right|\lesssim \phi(x) \quad\hbox{for all}\;x\in\mathbb{R}^n.
\end{equation}
\ebd

\bbd\label{lemma2}\cite[Lemma~2.4]{DaoReissig}
Let $s \in (0,1)$. Let $\psi$ be a smooth function satisfying $\partial_x^2\psi\in L^\infty(\mathbb{R}^n)$. For any $R>0$, let $\psi_R$ be a function defined by
$$ \psi_R(x):= \psi(x/R) \quad \text{ for all } x \in \R^n.$$
Then, $(-\Delta)^s (\psi_R)$ satisfies the following scaling properties:
$$(-\Delta)^s (\psi_R)(x)= R^{-2s}((-\Delta)^s\psi)(x/R) \quad \text{ for all } x \in \R^n. $$
\ebd
Using Lemmas \ref{lemma1} and \ref{lemma2}, or directly from \cite[Lemma~2.5]{DaoFino}, we may conclude easily the following result.
\bbd\label{lemma3}
Let $s \in (0,1]$, $R>0$ and $p>1$. Then, the following estimate holds
$$\int_{\mathbb{R}^n}(\phi_R(x))^{-\frac{1}{p-1}}\,\big|(-\Delta)^s\phi_R(x)\big|^{\frac{p}{p-1}}\, dx\lesssim R^{-\frac{2sp}{p-1}+n},$$
where $\phi_R(x):= \phi({x}/{R})$ and $\phi$ is given in \eqref{testfunction}.
\ebd


\section{Proof of the main result} \label{Sec.Proof}
Before starting our proof, we define weak solutions for (\ref{equation1}).
\begin{dn}\label{Def.weakSol}
 Let $T>0$, $p>1$, and $(u_0,u_1)\in L^2(\mathbb{R}^n)\times L^2(\mathbb{R}^n)$. A function $u$ is said to be a global weak solution to \eqref{equation1} if 
$$ u\in L^1_{loc}\big((0,\infty),L^2(\mathbb{R}^n)\big) \quad \text{ satisfying }\quad u_t\in L^p_{loc}\big((0,\infty),L^{2p}(\mathbb{R}^n)\big)\cap  L^1_{loc}\big((0,\infty),L^2(\mathbb{R}^n)\big), $$
and the following formulation holds
 	\begin{align*}
&\int_0^\infty\int_{\mathbb{R}^n}|u_t|^p\varphi(t,x)\,dx\,dt +\int_{\mathbb{R}^n}u_1(x)\varphi(0,x)\,dx \\
&\qquad =-\int_0^\infty\int_{\mathbb{R}^n}u_t(t,x)\varphi_{t}(t,x)\,dx\,dt- \int_0^\infty\int_{\mathbb{R}^n}u(t,x)\,\Delta\varphi(t,x)\,dx\,dt \\
&\qquad\quad + \mu\int_0^\infty\int_{\mathbb{R}^n}u_t(t,x)\,(-\Delta)^{\sigma/2}\varphi(t,x)\,dx\,dt
\end{align*}
for any test function $\varphi\in C\big([0,\infty);H^2(\mathbb{R}^n)\big) \cap C^1\big([0,\infty);L^2(\mathbb{R}^n)\big)$ such that its support in time is compact.
\end{dn}

\begin{proof}[Proof of Theorem \ref{Blow-up.Main}]
First, we introduce the function $\phi=\phi(x)$ as defined in \eqref{testfunction} with $s=\sigma/2$ and the function $\eta= \eta(t)$ having the following properties:
\begin{align}
&1.\quad \eta \in \mathcal{C}_0^\ity([0,\ity)) \text{ and }
\eta(t)=\begin{cases}
1 &\quad \text{ if }\quad 0 \le t \le \frac{1}{2}, \\
\text{decreasing } &\quad \text{ if }\quad \frac{1}{2} \le t \le 1, \\
0 &\quad \text{ if }\quad t \ge 1,
\end{cases} \nonumber \\
&2.\quad \eta^{-\frac{1}{p}}(t)\,|\eta'(t)| \le C \quad \text{ for any } t \in \Big[\frac{1}{2},1\Big]. \label{condition9.2.1}
\end{align}
For the existence of such function, see e.g. \cite[Chapter~1]{Yuta}. Let $R$ be a large parameter in $[0,\ity)$. We define the following test function:
$$ \varphi_R(t,x):= \eta_R(t) \phi_R(x), $$
where $\eta_R(t):= \eta\big(R^{-\bar{\sigma}}t\big)$ and $\phi_R(x):= \phi\big(R^{-1}K^{-1}x\big)$ for some $K\ge 1$ which will be fixed later. Moreover, we introduce the function
$$ \Psi_R(t)= \int_t^\ity \eta_R(\tau)d\tau \quad \text{for all}\,\,t\geq0. $$
Because of supp$\eta_R \subset \big[0,R^{\bar{\sigma}}\big]$, it follows supp$\Psi_R \subset \big[0,R^{\bar{\sigma}}\big]$. Here we also notice that the relation $\Psi'_R(t)= -\eta_R(t)$ holds. We define the functionals
$$ I_R:= \int_0^{\ity}\int_{\R^n}|u_t(t,x)|^p \varphi_R(t,x)\,dxdt= \int_0^{R^{\bar{\sigma}}}\int_{\R^n}|u_t(t,x)|^p \varphi_R(t,x)\,dxdt, $$
and
$$ I_{R,t}:= \int_{R^{\bar{\sigma}}/2}^{R^{\bar{\sigma}}}\int_{\R^n}|u_t(t,x)|^p \varphi_R(t,x)\,dxdt  \quad\text{ and }\quad I_{R,x}:= \int_0^{R^{\bar{\sigma}}}\int_{|x|\ge RK}|u_t(t,x)|^p \varphi_R(t,x)\,dxdt. $$
Let us assume that $u= u(t,x)$ is a global weak solution to (\ref{equation1}), then
\begin{align*}
I_R +\int_{\mathbb{R}^n}u_1(x)\phi_R(x)\,dx &=-\int_{R^{\bar{\sigma}}/2}^{R^{\bar{\sigma}}}\int_{\mathbb{R}^n}u_t(t,x)\eta'_R(t)\phi_R(x)\,dx\,dt \\
&\qquad + \int_0^{R^{\bar{\sigma}}}\int_{|x|\ge RK}u(t,x)\,\Psi'_R(t)\Delta\phi_R(x)\,dx\,dt \\
&\qquad + \mu\int_0^{R^{\bar{\sigma}}}\int_{\mathbb{R}^n}u_t(t,x)\,\eta_R(t)(-\Delta)^{\sigma/2}\phi_R(x)\,dx\,dt.
\end{align*}
Using integrating by parts, we conclude that 
\begin{align}
&I_R +\int_{\mathbb{R}^n}u_1(x)\phi_R(x)\,dx+\int_{\mathbb{R}^n}u_0(x)\Psi_R(0)\Delta\phi_R(x)\,dx \nonumber \\
&\quad =-\int_{R^{\bar{\sigma}}/2}^{R^{\bar{\sigma}}}\int_{\mathbb{R}^n}u_t(t,x)\eta'_R(t)\phi_R(x)\,dx\,dt-\int_0^{R^{\bar{\sigma}}}\int_{|x|\ge RK}u_t(t,x)\,\Psi_R(t)\Delta\phi_R(x)\,dx\,dt \nonumber \\
&\qquad + \mu\int_0^{R^{\bar{\sigma}}}\int_{\mathbb{R}^n}u_t(t,x)\,\eta_R(t)(-\Delta)^{\sigma/2}\phi_R(x)\,dx\,dt\nonumber \\
&\quad =: - J_1- J_2+ J_3. \label{t9.2.1}
\end{align}
Applying H\"{o}lder's inequality with $\frac{1}{p}+\frac{1}{p'}=1$ we may estimate $J_1$ as follows:
\begin{align*}
|J_1| &\le \int_{\frac{R^{\bar{\sigma}}}{2}}^{R^{\bar{\sigma}}}\int_{\R^n} |u_t(t,x)|\, \big|\eta'_R(t)\big| \phi_R(x) \, dx\,dt \\
&\lesssim \Big(\int_{{R^{\bar{\sigma}}}/{2}}^{R^{\bar{\sigma}}}\int_{\R^n} \Big|u_t(t,x)\varphi^{\frac{1}{p}}_R(t,x)\Big|^p \,dx\,dt\Big)^{\frac{1}{p}} \Big(\int_{{R^{\bar{\sigma}}}/{2}}^{R^{\bar{\sigma}}}\int_{\R^n} \Big|\varphi^{-\frac{1}{p}}_R(t,x) \eta'_R(t) \phi_R(x)\Big|^{p'}\, dx\,dt\Big)^{\frac{1}{p'}} \\
&\lesssim I_{R,t}^{\frac{1}{p}}\, \Big(\int_{{R^{\bar{\sigma}}}/{2}}^{R^{\bar{\sigma}}}\int_{\R^n} \eta_R^{-\frac{p'}{p}}(t) \big|\eta'	_R(t)\big|^{p'} \phi_R(x)\, dx\,dt\Big)^{\frac{1}{p'}}.
\end{align*}
By the change of variables $\tilde{t}:= R^{-\bar{\sigma}}t$ and $\tilde{x}:= R^{-1}K^{-1}x$, a straight-forward calculation gives
\begin{equation}
|J_1| \lesssim I_{R,t}^{\frac{1}{p}}\, R^{-\bar{\sigma}+ \frac{n+\bar{\sigma}}{p'}}K^{\frac{n}{p'}}\Big(\int_{\R^n} \big< \tilde{x}\big>^{-n-\sigma}\, d\tilde{x}\Big)^{\frac{1}{p'}}\lesssim I_{R,t}^{\frac{1}{p}}\, R^{-\bar{\sigma}+ \frac{n+\bar{\sigma}}{p'}}K^{\frac{n}{p'}}. \label{t9.2.2}
\end{equation}
Here we used $\eta'_R(t)= R^{-\bar{\sigma}}\eta'(\tilde{t})$ and the assumption (\ref{condition9.2.1}). Now let us turn to estimate $J_2$ and $J_3$. Applying H\"{o}lder's inequality again as we estimated $J_1$ leads to
$$|J_2|\le I_{R,x}^{\frac{1}{p}}\, \Big(\int_0^{R^{\bar{\sigma}}}\int_{|x|\ge RK} \Psi^{p'}_R(t)\eta^{-\frac{p'}{p}}_R(t) \phi^{-\frac{p'}{p}}_R(x)\, \big|\Delta \phi_R(x)\big|^{p'} \, dx\,dt\Big)^{\frac{1}{p'}}, $$
and
$$|J_3|\le I_R^{\frac{1}{p}}\, \Big(\int_0^{R^{\bar{\sigma}}}\int_{\R^n} \eta_R(t) \phi^{-\frac{p'}{p}}_R(x)\, \big|(-\Delta)^{\sigma/2}\phi_R(x)\big|^{p'} \, dx\,dt\Big)^{\frac{1}{p'}}. $$
In order to control $J_2$, we derive the following estimate:
\begin{align*}
\eta^{-\frac{p'}{p}}_R(t)\Psi^{p'}_R(t) = \eta^{-\frac{p'}{p}}_R(t)\Big( \int_t^\ity \eta_R(\tau)d\tau \Big)^{p'} &= \eta^{-\frac{p'}{p}}_R(t)\Big( \int_t^{R^{\bar{\sigma}}} \eta_R(\tau)d\tau \Big)^{p'} \\ 
&\le \eta^{-\frac{p'}{p}}_R(t) \eta_R(t)(R^{\bar{\sigma}}- t)^{p'}\le R^{\bar{\sigma} p'} \eta_R(t)\le R^{\bar{\sigma} p'},
\end{align*}
where we have used the fact that $\eta_R$ is a non-increasing function satisfying $\eta_R\le1$. Then, carrying out the change of variables $\tilde{t}:= R^{-\bar{\sigma}}t$, $\tilde{x}:= R^{-1}K^{-1}x$ and Lemma \ref{lemma3} with $s=1$ we arrive at
\begin{equation}
|J_2| \lesssim I_{R,x}^{\frac{1}{p}}\, R^{-2+\bar{\sigma}+ \frac{n+\bar{\sigma}}{p'}}K^{-2+ \frac{n}{p'}}. \label{t9.2.3}
\end{equation}
Next carrying out again the change of variables $\tilde{t}:= R^{-\bar{\sigma}}t$ and $\tilde{x}:= R^{-1}K^{-1}x$ and employing Lemma \ref{lemma2}, then Lemma \ref{lemma3}, with $s=\sigma/2$, we can proceed $J_3$ as follows:
\begin{equation}
|J_3| \lesssim I_R^{\frac{1}{p}}\, R^{-\sigma+ \frac{n+\bar{\sigma}}{p'}}K^{-\sigma+ \frac{n}{p'}}. \label{t9.2.4}
\end{equation}
Combining the estimates from (\ref{t9.2.1}) to (\ref{t9.2.4}) we may arrive at
\begin{align*}
&I_R+\int_{\R^n} u_1(x) \phi_R(x)\, dx \\
&\qquad \le C_0\Big(I_{R,t}^{\frac{1}{p}}R^{-\bar{\sigma}+ \frac{n+\bar{\sigma}}{p'}}K^{\frac{n}{p'}}+ I_{R,x}^{\frac{1}{p}}R^{-2+\bar{\sigma}+ \frac{n+\bar{\sigma}}{p'}}K^{-2+ \frac{n}{p'}}+ I_R^{\frac{1}{p}}R^{-\sigma+ \frac{n+\bar{\sigma}}{p'}}K^{-\sigma+ \frac{n}{p'}}\Big) \\
&\qquad \qquad+ \int_{\mathbb{R}^n}|u_0(x)|\Psi_R(0)|\Delta\phi_R(x)|\,dx.
\end{align*}
Moreover, it is clear that $\Psi_R(0)\leq R^{\bar{\sigma}}$. By the change of variables, using Lemma \ref{lemma1} we can easily check that $|\Delta\phi_R(x)|\leq R^{-2}\phi_R(x)$. Therefore, this implies that
\begin{align}
&I_R+\int_{\R^n} u_1(x) \phi_R(x)\, dx \nonumber \\
&\qquad \le C_0\Big(I_{R,t}^{\frac{1}{p}}R^{-\bar{\sigma}+ \frac{n+\bar{\sigma}}{p'}}K^{\frac{n}{p'}}+ I_{R,x}^{\frac{1}{p}}R^{-2+\bar{\sigma}+ \frac{n+\bar{\sigma}}{p'}}K^{-2+ \frac{n}{p'}}+ I_R^{\frac{1}{p}}R^{-\sigma+ \frac{n+\bar{\sigma}}{p'}}K^{-\sigma+ \frac{n}{p'}}\Big) \nonumber \\
&\qquad \qquad +R^{\bar{\sigma}-2}\int_{\mathbb{R}^n}|u_0(x)|\phi_R(x)\,dx. \label{t9.2.5}
\end{align}
Because of the assumption (\ref{blow-up.1}), there exists a sufficiently large constant $R_1> 0$ such that it holds
\begin{equation}
\int_{\R^n} u_1(x) \phi_R(x)\, dx >0 \label{t9.2.6}
\end{equation}
for all $R > R_1$. Since $u_0 \in L^1$, it implies immediately that
$$ R^{\bar{\sigma}-2}\int_{\mathbb{R}^n}|u_0(x)|\phi_R(x)\,dx \to 0 \quad \text{ as } R\to \ity. $$
Hence, from (\ref{t9.2.6}) there exists there exists a sufficiently large constant $R_2> 0$ such that
$$R^{\bar{\sigma}-2}\int_{\mathbb{R}^n}|u_0(x)|\phi_R(x)\,dx< \frac{1}{2}\int_{\R^n} u_1(x) \phi_R(x)\, dx $$
for all $R > R_2$. Now we choose $R_0:= \max\{R_1,\,R_2\}$. Then, from (\ref{t9.2.5}) we have
\begin{align}
&I_R+ \frac{1}{2}\int_{\R^n} u_1(x) \phi_R(x)\, dx \nonumber \\
&\qquad \le C_0\Big(I_{R,t}^{\frac{1}{p}}R^{-\bar{\sigma}+ \frac{n+\bar{\sigma}}{p'}}K^{\frac{n}{p'}}+ I_{R,x}^{\frac{1}{p}}R^{-2+\bar{\sigma}+ \frac{n+\bar{\sigma}}{p'}}K^{-2+ \frac{n}{p'}}+ I_R^{\frac{1}{p}}R^{-\sigma+ \frac{n+\bar{\sigma}}{p'}}K^{-\sigma+ \frac{n}{p'}}\Big) \label{*} 
\end{align}
for all $R > R_0$. By choosing $K=1$ and noticing the relations $I_{R,t}\le I_R$ and $I_{R,x}\le I_R$ we may arrive, particularly, at
\begin{equation}
I_R+ \frac{1}{2}\int_{\R^n} u_1(x) \phi_R(x)\, dx \le C_0\, I_R^{\frac{1}{p}}\, R^{-\bar{\sigma}+ \frac{n+\bar{\sigma}}{p'}} \label{t9.2.7}
\end{equation}
for all $R > R_0$. Thanks to the following $\varepsilon$-Young's inequality:
$$ab\leq \varepsilon a^p+C(\varepsilon)b^{p'}\quad \text{ for all }a,b>0 \text{ and for any }\varepsilon>0, $$ 
we conclude
$$ C_0\, I_R^{\frac{1}{p}}\, R^{-\bar{\sigma}+ \frac{n+\bar{\sigma}}{p'}}\le \frac{1}{2}I_R+ C_1\,R^{-\bar{\sigma}p'+ n+\bar{\sigma}}.$$
Consequently, from (\ref{t9.2.7}) we derive
$$ \frac{1}{2}I_R+ \frac{1}{2}\int_{\R^n} u_1(x) \phi_R(x)\, dx\le C_1\, R^{-\bar{\sigma}p'+ n+\bar{\sigma}}, $$
which follows that
\begin{align}
I_R &\le 2C_1\, R^{-\bar{\sigma}p'+ n+\bar{\sigma}}, \label{t9.2.8} \\
\int_{\R^n} u_1(x) \phi_R(x)\,dx &\le 2C_1\, R^{-\bar{\sigma}p'+ n+\bar{\sigma}}, \label{t9.2.9}
\end{align}
for all $R > R_0$. It is clear that the assumption (\ref{blow-up.2}) is equivalent to $-\bar{\sigma} p'+ n+\bar{\sigma}\le 0$. For this reason, let us now separate our considerations into two cases as follows.\\

\noindent \textbf{Case 1}: $-\bar{\sigma} p'+ n+\bar{\sigma}< 0$, i.e. the subcritical case. Letting $R \to \ity$ in (\ref{t9.2.9}) 
we infer a contradiction to (\ref{blow-up.1}).\\

\noindent \textbf{Case 2}: $-\bar{\sigma} p'+ n+\bar{\sigma}= 0$, i.e. the critical case. Then, we can see from (\ref{t9.2.8}) that $I_R \le 2C_1$ for all $R > R_0$. Using Beppo Levi's theorem on monotone convergence, on the one hand, we derive
$$\int_0^\infty\int_{\mathbb{R}^n}|u_t(t,x)|^p\,dx\,dt=\lim_{R\rightarrow\infty}\int_0^{R^{\bar{\sigma}}}\int_{\mathbb{R}^n}|u_t(t,x)|^p\,\varphi_R(t,x)\,dx\,dt= \lim_{R\rightarrow\infty}I_R \leq 2C_1, $$
that is, $u_t \in L^p((0,\infty)\times\mathbb{R}^n)$. By the absolute continuity of the Lebesgue integral, it follows that $I_{R,t} \to 0$ and $I_{R,x} \to 0$ as $R\to \ity$. On the other hand, using again the fact that $p=1+ \frac{\bar{\sigma}}{n}$ we obtain from (\ref{*}) the following estimate:
\begin{equation}
I_R+ \frac{1}{2}\int_{\R^n} u_1(x) \phi_R(x)\, dx \le C_0\Big(I_{R,t}^{\frac{1}{p}}K^{\frac{n}{p'}}+ I_{R,x}^{\frac{1}{p}}R^{-2+2\bar{\sigma}}K^{-2+ \frac{n}{p'}}+ I_R^{\frac{1}{p}}R^{-\sigma+ \bar{\sigma}}K^{-\sigma+ \frac{n}{p'}}\Big) \label{t9.2.10}
\end{equation}
 for all $K\geq1$ and all $R > R_0$. \\
$\bullet$ If $\sigma\in(0,1]$, then $\bar{\sigma}=\sigma$. As a consequence, from (\ref{t9.2.10}) we have
\begin{equation}
I_R+ \frac{1}{2}\int_{\R^n} u_1(x) \phi_R(x)\, dx \le C_0\Big(I_{R,t}^{\frac{1}{p}}K^{\frac{n}{p'}}+ I_{R,x}^{\frac{1}{p}}R^{-2(1-\sigma)}K^{-2+ \frac{n}{p'}}+ I_R^{\frac{1}{p}}K^{-\sigma+ \frac{n}{p'}}\Big) \label{t9.2.11}
\end{equation}
for all $K\geq1$ and all $R > R_0$. Letting $R \to \ity$ in (\ref{t9.2.11}) we get
\begin{equation}
\int_{\R^n} u_1(x)dx \lesssim K^{-\sigma+ \frac{n}{p'}} \quad\text{for all}\,\,K\geq1.\label{t9.2.12}
\end{equation}
Due to $p=1+ \frac{\bar{\sigma}}{n}=1+ \frac{\sigma}{n}$, it is clear that $-\sigma+ \frac{n}{p'}=-\frac{\sigma}{p'}<0$. Therefore, we can fix a sufficiently large constant $K\ge 1$ in (\ref{t9.2.12}) to gain a contradiction to (\ref{blow-up.1}). \\
$\bullet$ If $\sigma\in(1,2]$, then $\bar{\sigma}=1$. As a result, choosing $K=1$ we may conclude from (\ref{t9.2.10}) that
\begin{equation}
\int_{\R^n} u_1(x) \phi_R(x)\, dx\le 2C_0\Big(I_{R,t}^{\frac{1}{p}}+ I_{R,x}^{\frac{1}{p}}+ I_R^{\frac{1}{p}}R^{1-\sigma}\Big) \label{**}
\end{equation}
for all $R > R_0$. Since $\sigma>1$, letting $R \to \ity$ in (\ref{**}) we obtain a contradiction to (\ref{blow-up.1}) again. \\

\noindent Summarizing, the proof Theorem \ref{Blow-up.Main} is completed.
\end{proof}


\end{document}